\documentclass[11pt]{amsart}
\usepackage{amsmath,amssymb,mathrsfs,xcolor,comment}
\usepackage{enumitem}
\usepackage{hyperref}
\usepackage{graphicx}
\hypersetup{
    colorlinks=true, %set true if you want colored links
    linktoc=all,     %set to all if you want both sections and subsections linked
    linkcolor=blue,  %choose some color if you want links to stand out
}

\newtheorem{theorem}{Theorem}[section]
\newtheorem{proposition}[theorem]{Proposition}

\newtheorem{lemma}[theorem]{Lemma}
\newtheorem{conjecture}[theorem]{Conjecture}
\theoremstyle{definition}
\newtheorem{definition}[theorem]{Definition}
\newtheorem{remark}[theorem]{Remark}

\numberwithin{equation}{section}

\newcommand{\pr}{\partial}
\renewcommand\div{\operatorname{div}}
\newcommand{\tr}{\operatorname{tr}}
\newcommand{\diam}{\operatorname{diam}}
\newcommand{\curv}{\kappa}
\DeclareMathOperator{\sn}{sn}
\DeclareMathOperator{\cs}{cs}
\DeclareMathOperator{\ct}{ct}
\DeclareMathOperator{\tn}{tn}

\begin{document}

\title[The prescribed point area estimate in constant curvature]{The prescribed point area estimate for minimal submanifolds in constant curvature}
\author{Keaton Naff}
\address{Department of Mathematics, Massachusetts Institute of Technology, Cambridge, MA 02139, USA}
\email{kn2402@mit.edu}

\author{Jonathan J. Zhu}
\address{Department of Mathematics, University of Washington, Seattle, WA, USA}
\email{jonozhu@uw.edu}

\begin{abstract}
We prove a sharp area estimate for minimal submanifolds that pass through a prescribed point in a geodesic ball in hyperbolic space, in any dimension and codimension. In certain cases, we also prove the corresponding estimate in the sphere. Our estimates are analogous to those of Brendle and Hung in the Euclidean setting.
\end{abstract}
\maketitle
%\tableofcontents

\section{Introduction}

In this note, we study the area of minimal submanifolds in spaces of constant curvature. Consider a space form $M \in \{\mathbb{H}^n, \mathbb{R}^n, \mathbb{S}^n\}$ which has constant curvature $\curv \in \{-1, 0, 1\}$ respectively. Let $B^n_R$ denote a geodesic ball in $M$ of radius $R \in (0, \frac{1}{2}\mathrm{diam}(M))$ and centre $o \in M$. Suppose $\Sigma \subset B^n_R$ is a $k$-dimensional minimal submanifold which passes through a point $y \in B^n_R$ and satisfies $\partial \Sigma \subset \partial B^n_R$. One of the simplest examples of such a minimal submanifold is a totally geodesic $k$-dimensional disk $B^k_{\underline{r}(y)} \subset B^n_R$ which passes orthogonally through $y$. In particular, $B^k_{\underline{r}(y)}$ is centred at $y$, which is its closest point to $o$, and has radius $\underline{r}(y)$ which may be explicitly given via the Pythagorean theorem - see (\ref{underline-r}). It is natural to ask if these totally geodesic disks have least area in $B^n_R$ among all such minimal submanifolds passing through $y$. 

Recently, using a very beautiful and simple variational argument, Brendle and Hung \cite{BH17} answered this question in the affirmative for minimal submanifolds of Euclidean space (in arbitrary dimension and codimension). 

\begin{theorem}[\cite{BH17}]\label{Brendle-Hung}
Let $B^n_R$ be an $n$-dimensional Euclidean ball of radius $R \in (0, \infty)$ and with centre $o \in \mathbb{R}^n$. Suppose $\Sigma$ is a $k$-dimensional minimal submanifold in $B^n_{R}$ which  passes through $y \in B^n_R$ and satisfies $\partial \Sigma \subset \partial B^n_R$. Then 
\[
|\Sigma| \geq |B^k_{\underline{r}(y)}|, 
\]
where $\underline{r}(y) = (R^2 - d(o,y)^2)^{\frac{1}{2}}$. Moreover, equality holds if and only if $\Sigma$ is a totally geodesic $k$-dimensional disk of radius $\underline{r}(y)$ orthogonal to $y$. 
\end{theorem}

Motivated by this result, we will prove the analogous area estimate for minimal submanifolds passing through a prescribed point in hyperbolic space:

\begin{theorem}\label{hyperbolic-space}
Let $B^n_R$ be an $n$-dimensional hyperbolic (geodesic) ball of radius $R \in (0, \infty)$ and centre $o \in \mathbb{H}^n$. Suppose $\Sigma$ is a $k$-dimensional minimal submanifold in $B^n_{R}$ which passes through a point $y \in B^n_{R}$ and satisfies $\partial \Sigma \subset \partial B^n_{R}$. Then
\[
|\Sigma| \geq |B^k_{\underline{r}(y)}|, 
\]
where $\underline{r}(y) = \cosh^{-1}\big(\frac{\cosh(R)}{\cosh(d(o, y))}\big)$. Moreover, equality holds if and only if $\Sigma$ is a totally geodesic $k$-dimensional disk of radius $\underline{r}(y)$ orthogonal to $y$. 
\end{theorem}

It remains to consider whether the sharp area estimate holds in the sphere. In this direction, we are able to prove the following partial result. 

\begin{theorem}\label{sphere}
Let $B^n_R$ be an $n$-dimensional spherical (geodesic) ball of radius $R \in (0, \frac{\pi}{2})$ and centre $o \in \mathbb{S}^n$. Suppose $\Sigma$ is a $k$-dimensional minimal submanifold in $B^n_{R}$ which passes through a point $y \in B^n_{R}$ and satisfies $\partial \Sigma \subset \partial B^n_{R}$. 
Further assume that either:
\begin{enumerate}[label=(\alph*)]
\item $k=1$; or
\item $ \cos(d(o,y) + R) \geq \sqrt{\frac{2}{k}}.$
\end{enumerate}
Then
\[
|\Sigma| \geq |B^k_{\underline{r}(y)}|, 
\]
where $\underline{r}(y) = \cos^{-1}\big(\frac{\cos(R)}{\cos(d(o, y))}\big)$. Moreover, equality holds if and only if $\Sigma$ is a totally geodesic $k$-dimensional disk of radius $\underline{r}(y)$ orthogonal to $y$. 
\end{theorem} 

Condition (b) is not the most general condition under which we can prove the sharp estimate (see for instance Lemma \ref{lem:divergence-sph}). However, we have chosen it for simplicity, as there do appear to be geometric obstructions to the success of our method in the sphere (see Section \ref{sec:wedge}). Note that condition (b) can only hold if $k>2$. 

When $y = o$ lies at the centre of the geodesic ball, the sharp area estimate is well-known and follows from an important and (more general) monotonicity formula for minimal submanifolds (see Theorem \ref{thm:classic-monotonicity} below). The prescribed point question for $y \neq o$ was first raised and studied in three-dimensional Euclidean space by Alexander and Osserman in \cite{AO75}, and later in higher dimensions with Hoffman in \cite{AHO74}. They conjectured the estimate of Theorem \ref{Brendle-Hung} and managed to prove that estimate under certain restrictions on the dimension and topology of $\Sigma$, before Brendle and Hung later proved it in full generality. Additionally, they observed that the sharp prescribed-point area estimate follows from the sharp isoperimetric inequality for minimal submanifolds, which was also recently proven by Brendle in codimensions up to 2 \cite{Br21}. A related estimate for Gaussian measures of holomorphic fibres was discussed by Klartag \cite{K18}.

To prove Theorems \ref{hyperbolic-space} and \ref{sphere}, we will use what we call the ``vector field approach". The essential idea is as follows. Suppose $W$ is a smooth vector field on $B^n_R \setminus\{y\}$ with a pole at $y$ and suppose $\Sigma \subset B^n_R$ is $k$-dimensional and minimal with $y \in \Sigma$ and $\partial \Sigma \subset \partial B^n_R$. Applying the divergence theorem to $W$ on $\Sigma \setminus B^n_t(y)$ and sending $t \to 0$, one obtains 
\[
|\Sigma| \sup_{B^n_R} \mathrm{div}_{\Sigma}(W) \geq \int_{\partial \Sigma} \langle W, \nu_{\Sigma}\rangle - \mathrm{res}_{k-1}(W, y) |\mathbb{S}^{k-1}|,
\]
where $\mathrm{res}_{k-1}(W, y) := \lim_{x \to y} r_y(x)^{k-1} \langle W(x), \nabla r_y(x) \rangle$ is a ``$(k-1)$-residue" of the vector field at $y$ (see the proof of Proposition \ref{proof-assuming-vs} for more details). This inequality yields a lower bound for the area of $\Sigma$, provided one can construct a special vector field with a divergence upper bound, boundary behavior such that $\int_{\partial \Sigma} \langle W, \nu_\Sigma \rangle$ vanishes, and prescribed (negative) residue (of order $k-1$). Constructing a suitable $W$ is the essential difficulty in the proof - balancing these three conditions when constructing $W$ is rather delicate, especially if one hopes to obtain a sharp estimate. 

The approach outlined above was likely known to experts to yield the classical area estimate (Proposition \ref{classical-area-estimate}), where it is easy to check that on $B^n_1\setminus\{0\}$ in $\mathbb{R}^n$ the vector field
\[
W(x) = \frac{x}{k} - \frac{1}{k}\frac{x}{|x|^k} 
\]
has all of the desired properties. Our inspiration for using the vector field approach here comes from its recent successes in the free-boundary \cite{Br12} and prescribed-point \cite{BH17} settings in Euclidean space. 

In the prescribed-point setting, for $k \geq 3$ ($M = \mathbb{R}^n$, $R = 1$), Brendle and Hung constructed the vector field
\begin{equation} \label{brendle-hung-vector}
W(x) =\frac{1}{k} \left( 1- \frac{(1- 2\langle x,y\rangle +|y|^2)^{\frac{k}{2}}}{|x-y|^k} \right) (x-y)  + \frac{1}{k-2} \left( \frac{(1- 2\langle x,y\rangle +|y|^2)^{\frac{k-2}{2}}}{|x-y|^{k-2}} -1\right) y.
\end{equation}
This vector field $W$ is somewhat mysterious, as it does not have a simple geometric interpretation - in particular, it cannot be a gradient field. We will construct in Section \ref{sec:construction} a vector field analogous to $W$ in the sphere and hyperbolic space, and verify that our vector field satisfies the properties needed above to yield the sharp prescribed-point area estimate. Our vector field also reduces to Brendle and Hung's vector field $W$ in the Euclidean setting.

The difficulties addressed in our construction are rather subtle, and are masked in Euclidean space by the very simple parallel translation. In particular, one immediate issue is how to generalise the constant vector field $y$. Our solution is to use the generator of the translation isometry in the $y$ direction, and the (nonconstant) length of this field must be balanced in a precise manner in order for the divergence to satisfy a useful inequality. We remark that this construction seems to depend on the special geometry of these symmetric spaces, in the sense that the natural analogues in manifolds with curvature conditions seem to require comparisons in both directions. We also remark that a quite reasonable ansatz based on harmonic functions does not appear to bear fruit (see Remark \ref{rmk:super}). 

 The vector field approach was also used by Brendle in \cite{Br12} to prove a sharp area estimate for free-boundary minimal submanifolds in Euclidean balls. Subsequently, Freidin and McGrath \cite{FM19, FM20} were able to extend the free-boundary area estimate to minimal submanifolds of dimension $k\in\{2, 4, 6\}$ in spherical caps. However, the vector field approach does not seem to readily apply to the free-boundary problem in hyperbolic space. This apparently stands in contrast to our prescribed-point problem, in which hyperbolic space seems to be more amenable (although we also have partial success in the sphere).

Finally, we note that the second author \cite{Zhu18} showed Theorem \ref{Brendle-Hung} can be realised as a consequence of a more general `moving-centre' monotonicity formula. This is analogous to the fact that the centred area estimate, Proposition \ref{classical-area-estimate} below, is a consequence of the fixed-centre monotonicity formula Theorem \ref{thm:classic-monotonicity}. It turns out that both Theorems  \ref{hyperbolic-space} and \ref{sphere} also follow from certain weighted monotonicity formulae, somewhat analogous to the moving-centre monotonicity formula in \cite{Zhu18}. The details can be found in \cite{NZ22b}. 

The remainder of this paper is organised as follows. In Section 2, we discuss some background and preliminaries about space forms. In Section 3, we will give a proof of Theorems \ref{hyperbolic-space} and \ref{sphere}(b). In Section \ref{sec:geodesics}, by a separate argument, we verify the prescribed-point estimate for geodesics, in particular Theorem \ref{sphere}(a). Finally, in Section \ref{sec:other-domains}, we discuss the vector field method on domains other than balls. 

\subsection*{Acknowledgements}

JZ was supported in part by the National Science Foundation under grant DMS-1802984. KN was supported by the National Science Foundation under grant DMS-2103265. The authors would like to thank Jacob Bernstein for pointing out \cite{CG92} and for other interesting discussions.

\section{Preliminaries} 

In the following, we consider a space form $M \in \{\mathbb{H}^n, \mathbb{R}^n, \mathbb{S}^n\}$ of dimension $n \geq 2$ equipped with the metric $g$ of constant curvature $\curv \in \{-1, 0, 1\}$ respectively. We let $k \in \{1, \dots, n-1\}$. 

\subsection{Notations} Let $d(x, y)$ denote the distance between points $x, y \in M$ and for any $t > 0$, let $B_t^n = B_t^n(o) = \{x \in M : d(o, x) < t\} \subset M$ denote the geodesic ball of radius $t$ around a fixed point $o \in M$, which we call the origin. Throughout the following sections, we fix some $R \in (0, \frac{1}{2} \mathrm{diam}(M))$ and some $y \in B^n_R$, and let $\Sigma \subset B_R^n$ denote a $k$-dimensional (smooth) minimal submanifold which passes through $y$ and satisfies $\partial \Sigma \subset \partial B_R^n$. 

We let  
\[
\sn(r) := \begin{cases}\sinh(r) ,& M = \mathbb{H}^n \\ r ,& M = \mathbb{R}^n \\ \sin(r) ,& M = \mathbb{S}^n  \end{cases}, 
\]
denote the usual warping function of the metric $g = dr^2 + \sn(r)^2 g_{\mathbb{S}^{n-1}}$ and define 
\[
\cs(r) := \sn'(r) = \begin{cases} \cosh(r) ,& M = \mathbb{H}^n \\ 1 ,& M = \mathbb{R}^n \\ \cos(r) ,& M = \mathbb{S}^n  \end{cases},
\]
as well as $\tn(r) := \sn(r)/\cs(r)$ and $\ct(r) =1/\tn(r)$. Note that \begin{equation}\cs'(r) = -\curv\sn(r)\end{equation} and \begin{equation}\cs(r)^2 +\curv \sn(r)^2 = 1.\end{equation} 

We recall the Pythagorean theorem for distances. If $xyz$ is a geodesic triangle in $M$ with a right angle at $y$, then 
\begin{equation}
\begin{cases} \cs(d(x, z) )= \cs(d(x,y)) \cs(d(y,z)) ,& M \in \{\mathbb{H}^n, \mathbb{S}^n\} \\ d(x,z)^2 = d(x,y)^2 + d(y,z)^2 ,& M = \mathbb{R}^n \end{cases}. 
\end{equation}
Note that the Pythagorean theorem for $\curv = 0$ follows as a limiting case (at second order) of the Pythagorean theorem for $\curv \neq 0$. Note also, that when $\curv = 1$, it is possible for $\cs(d(x,z))$ to be negative. Nevertheless, the Pythagorean theorem still holds in such cases.

Given $z \in M$, we introduce the shorthand  $r_z(x) = d(x, z)$ for the distance function on $M$. Away from $z$ and the cut locus of $z$ the function $r_z$ is smooth. There, we have $|\nabla r_z| = 1$ and the Hessian is given by
\begin{equation}\label{Hessian-of-r}
\nabla^2 r_z = \ct(r_z) \big(g - dr_z \otimes dr_z\big). 
\end{equation}
For our origin $o \in M$, we will write $r(x)$ in place of $r_o(x)$. 

With some notation introduced, the radius $\underline{r}(y)$ from the introduction is given by 
\begin{equation}\label{underline-r}
\underline{r}(y) := \begin{cases} \cs^{-1}\big(\frac{\cs(R)}{\cs(r(y))}\big) ,& M \in \{\mathbb{H}^n, \mathbb{S}^n\} \\ (R^2 - r(y)^2)^\frac{1}{2} ,& M = \mathbb{R}^n \end{cases}.
\end{equation}
Here the inverses are defined as functions $\cosh^{-1} : [1, \infty) \to [0, \infty)$, $(\cdot)^\frac{1}{2} : [0, \infty) \to [0, \infty)$, and $\cos^{-1} : (-1, 1] \to [0, \pi)$. If $r(y)=0$, we understand that $\underline{r}(y)=R$. Note that $\underline{r}(y)$ is well-defined as $r(y) \in [0, R)$ and $R \in [0, \frac{1}{2} \mathrm{diam}(M))$. 

For $r \in [0, \frac{1}{2}\mathrm{diam}(M))$, define 
\begin{equation}
A(r) := \int_0^r \mathrm{sn}(t)^{k-1} \, dt. 
\end{equation}
For $k\geq 2$ and $(M,k) \neq ( \mathbb{R}^n,2)$, define
\begin{equation}
G(r) :=- \int_r^{\frac{1}{2}\mathrm{diam}(M)} \frac{1}{A'(t)} \, dt. 
\end{equation}
Note that this latter integral converges for $r \in (0, \frac{1}{2}\mathrm{diam}(M))$, so $G(r)$ is well-defined. When $k = 2$ and $M = \mathbb{R}^n$, we instead define $G(r) = \log(r)$; finally when $k=1$ define $G(r)=r$. In all cases, we have $G'(r) = \frac{1}{A'(r)}$; hence $G(r)$ is monotone increasing. For $M \in \{\mathbb{H}^n, \mathbb{R}^n\}$, up to a constant the function $G(r)$ is the fundamental solution of the Laplacian; for $M = \mathbb{S}^n$, the function $G(r)$ is related, but has poles at both the origin and its antipodal point. Also, observe that $A(r)$ is positive and increasing. Moreover, \begin{equation}\frac{A''(r)}{A'(r)} = (k-1)\ct(r).\end{equation} The area of a $k$-dimensional totally geodesic disk in the space form $M$ is 
\begin{equation}
|B^k_r| = A(r) |\mathbb{S}^{k-1}|. 
\end{equation}
Finally, we note that as $r \to 0$, these functions have the asymptotics
\begin{equation}\label{asymptotics-1}
A'(r) = r^{k-1} + o(r^{k-1}),
\end{equation}
\begin{equation}\label{asymptotics-2}
A(r) = \frac{1}{k} r^k + o(r^{k}), 
\end{equation}
and 
\begin{equation}\label{asymptotics-3}
G(r) = \begin{cases} -\frac{1}{k-2}r^{2-k} + o(r^{2-k}), & k\neq 2  \\  \log(r) + o(\log(r)) ,& k=2\end{cases}. 
\end{equation}

\subsection{Projection and divergence of vector fields}

Given a $k$-plane $S\subset T_x M$ and a vector field $W$, we define $W^\top$ to be the projection of $W$ to $S$, and $W^\perp$ the projection to the orthogonal complement, so that $W= W^\top + W^\perp$. If $f$ is a function on $M$, then we also set $\nabla^\top f = (\nabla f)^\top$ and $\nabla^\perp f = (\nabla f)^\perp$. We define 
\begin{equation}
\div_S W (x) = \tr_{S} \nabla W(x) = \sum_{i=1}^k g(\nabla_{e_i}W(x), e_i),
\end{equation}
where $\{e_1, \dots, e_k\}$ is any orthonormal basis of $S$. If $\Sigma$ is a smooth $k$-dimensional submanifold, then we implicitly take $S=T_x\Sigma$ and write $\div_\Sigma W (x) = \div_{T_x\Sigma} W(x)$. 

Note that \begin{equation} \div_\Sigma W  = \div_\Sigma W^\top - g(\vec{H}, W^\perp),\end{equation}
In particular, if $\Sigma$ is minimal, then $\div_\Sigma W = \div_\Sigma W^\top$. 

\subsection{The classical area estimate when $y = o$} 
\label{sec:classical}

Recall $r=r_o$ and consider the vector fields
\begin{equation}
W_0 := \frac{1}{A'(r)} \nabla r  = \nabla G(r)
\end{equation}
and 
\begin{equation}
W_1 := \frac{A(r)}{A'(r)} \nabla r = A(r) \nabla G(r).
\end{equation}
In the following, for any vector field $W$ we write $W = W^\top + W^\perp$ to denote the tangential and orthogonal components of $W$ along $\Sigma$. Correspondingly, let $\nabla^\top r = (\nabla r)^\top$ and $\nabla^\perp r = (\nabla r)^\perp$. The divergence of these vector fields $W_0$ and $W_1$ is a classical computation and plays an important role in classical monotonicity formulae and related area estimates for minimal submanifolds. 

\begin{proposition} \label{divergences}
Given a $k$-plane $S \subset T_xM$, where $0 < r(x) < \frac{1}{2}\mathrm{diam}(M)$, we have
\begin{equation}
\mathrm{div}_{S}W_0(x)  = k \frac{1}{A'(r)}\ct(r) |\nabla^\perp r|^2, 
\end{equation}
and 
\begin{equation}
\mathrm{div}_{S}W_1(x) =  1 - \Big(1 - k\frac{A(r)}{A'(r)} \ct(r)\Big)|\nabla^\perp r|^2.
\end{equation}
\end{proposition} 
\begin{proof}
Fix any point in $x$ with $0 < r(x) < \frac{1}{2} \mathrm{diam}(M)$ and let $\{e_1, \dots, e_k\}$ be a orthonormal basis of $S \subset T_xM$.   For a radial vector field $ \phi(r)\nabla r$ we have 
\[\div_S\left( \phi(r)\nabla r\right) = \phi'(r) |\nabla^\top r|^2 + \phi(r) \tr_S \nabla^2 r. \]
By \eqref{Hessian-of-r}, we have
\[
\tr_S \nabla^2 r = \ct(r) ( k-|\nabla^\top r|^2). 
\]
Since $|\nabla^\top r|^2= 1 - |\nabla^\perp r|^2$, we then have  
\[
\div_S \left( \phi(r)\nabla r\right) = \left( (k-1)\phi(r)\ct(r)+ \phi'(r) \right) - \left( \phi'(r) - \phi(r)\ct(r)\right) |\nabla^\perp r|^2.
\]
The proposition follows after noting that \[\left(\frac{1}{A'(r)}\right)' = -\frac{A''(r)}{A'(r)^2} = -(k-1)\ct(r) \frac{1}{A'(r)},\] and \[\left(\frac{A(r)}{A'(r)}\right)' =1 -\frac{A(r)A''(r)}{A'(r)^2} = 1-(k-1)\ct(r)\frac{A(r)}{A'(r)}. \]
\end{proof}

It is straightforward to verify\footnote{If $f(r)=1 -k  \frac{A(r)}{A'(r)}\ct(r)$, then $(\sn(r)^kf(r))' =  \curv k\sn(r) A(r)$ and $f(0) = 0$. Hence $\mathrm{sign}(f(r)) = \kappa$.} that $\div_S W_1 =1$ if $M=\mathbb{R}^n$; $\div_S W_1 \geq 1$ if $M=\mathbb{H}^n$, and $1\geq \div_S W_1 \geq |\nabla^\top r|^2$ if $M=\mathbb{S}^n$. These \textit{lower} bounds on $\div_S W_1$ may be used to show the classical monotonicity formulae (see \cite{And82} for $\mathbb{H}^n$, \cite{GS87} for $\mathbb{S}^n$, \cite{Si83} for instance for $\mathbb{R}^n$):

\begin{theorem}
\label{thm:classic-monotonicity}
Suppose $M \in \{\mathbb{H}^n, \mathbb{R}^n, \mathbb{S}^n\}$. Let $\Sigma \subset B^n_R$ be a $k$-dimensional minimal submanifold in a geodesic ball of radius $R \in (0, \frac{1}{2}\mathrm{diam}(M))$ with $\partial \Sigma \subset \partial B^n_R$. Define
\[
Q_A(t) := \frac{|\Sigma \cap B^n_t|}{|B^k_t|}, \quad Q_I(t) :=  \frac{1}{|B^k_t|} \int_{\Sigma \cap B^n_t} |\nabla^\top r|^2. 
\]
If $M \in \{\mathbb{H}^n, \mathbb{R}^n\}$, then $t \mapsto Q_I(t)$ and $t \mapsto Q_A(t)$ are both monotone increasing for $t \in (0, R)$.  If $M = \mathbb{S}^n$, then $t \mapsto Q_I(t)$ is monotone increasing for $t \in (0, R)$. In each case, the monotone quantity is constant if and only if $\Sigma$ is a totally geodesic disk.
\end{theorem}

Similarly, so long as $\cs(r) \geq 0$, one has the lower bound $\div_S W_0 \geq 0$. This can be used to derive another monotone quantity along the boundaries of centred geodesic balls. This monotonicity is equivalent to the monotonicity of $Q_A(t)$ in Theorem \ref{thm:classic-monotonicity} when $M = \mathbb{R}^n$, but is different for $M \in \{\mathbb{H}^n, \mathbb{S}^n\}$. This monotonicity was essentially observed by Choe and Gulliver in \cite{CG92}. 

\begin{theorem}
\label{thm:bdry-monotonicity}
Suppose $M \in \{\mathbb{H}^n, \mathbb{R}^n, \mathbb{S}^n\}$. Let $\Sigma \subset B^n_R$ be a $k$-dimensional minimal submanifold in a geodesic ball of radius $R \in (0, \frac{1}{2}\mathrm{diam}(M))$ with $\partial \Sigma \subset \partial B^n_R$. Define \[Q_\pr (t):= \frac{1}{|\pr B^k_t|} \int_{\Sigma \cap \pr B^n_t} |\nabla^\top r|.\] Then $t \mapsto Q_\pr(t)$ is monotone increasing for $t \in (0, R)$ and is constant if and only if $\Sigma$ is a totally geodesic disk.
\end{theorem}

These classical monotonicity formulae and their proofs are discussed in more detail in \cite{NZ22b}. If $\Sigma$ contains the origin $o$, then we have the classical area estimate for minimal submanifolds through the centre of the ball: 

\begin{proposition}\label{classical-area-estimate}
Suppose $M \in \{\mathbb{H}^n, \mathbb{R}^n, \mathbb{S}^n\}$. Let $\Sigma \subset B^n_R$ be a $k$-dimensional minimal submanifold in a geodesic ball of radius $R \in (0,\mathrm{diam}(M))$ with $\partial \Sigma \subset \partial B^n_R$. If $o \in \Sigma$, then 
\[
|\Sigma| \geq |B^k_R|. 
\] 
Moreover, equality holds if and only if $\Sigma$ is a totally geodesic disk.
\end{proposition}

For $R \in (0, \frac{1}{2} \mathrm{diam}(M))$, Proposition \ref{classical-area-estimate} follows either from the monotonicities in Theorem \ref{thm:classic-monotonicity}, or from the one in Theorem \ref{thm:bdry-monotonicity} and the coarea formula. Note that the monotonicity formulae in the sphere cannot hold past $R = \frac{1}{2} \mathrm{diam}(M)$. In fact, there is a third proof of Proposition \ref{classical-area-estimate}, which holds for all $R \in (0, \mathrm{diam}(M))$, using the divergence theorem applied to the combination $W_1- A(R)W_0$. This latter proof is one of the starting points for our construction. 

\subsection{The functions $s$ and $\rho$} 
In the previous section, we saw that Theorems \ref{hyperbolic-space} and \ref{sphere} are classical consequences of more general monotonicity formula when $y = o$. For the remainder of the paper, it suffices to assume $y \neq o$. We will now define two functions $s$ and $\rho$ that are important to the proof of Theorems \ref{hyperbolic-space} and \ref{sphere}. 

\begin{figure}[h]
\centering
\includegraphics[width=0.5\textwidth]{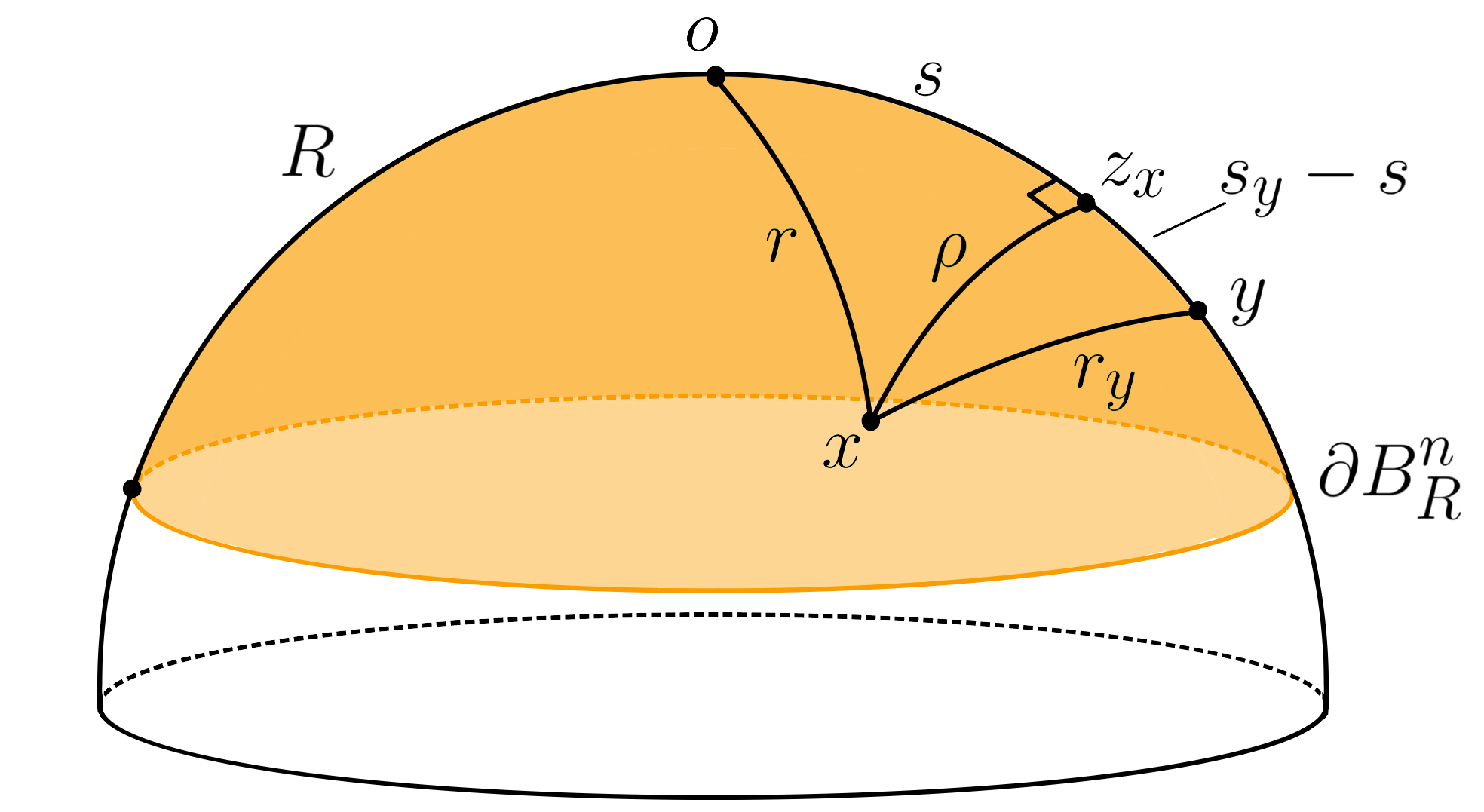}
\caption{Representation of the functions $s$ and $\rho$ on $\mathbb{S}^n$, including the right geodesic triangles $oz_x x$ and $yz_x x$. }
\label{fig:s-and-rho}
\end{figure}

Let $\gamma \subset M$ be the unique maximal geodesic containing the points $o$ and $y$. Let $\rho(x) :=\inf_{z \in \gamma} d(x, z)$ denote the distance to the geodesic $\gamma$. For each point $x$ with $\rho(x)<\frac{1}{2}\diam(M)$, there exists a unique point $z_x \in \gamma$ such that $\rho(x) = d(x, z_x)$. Note that $\diam(M)<\infty$ only when $M=\mathbb{S}^n$. In this case $\{\rho(x) = \frac{1}{2}\diam(M)\}$ consists of a copy of $\mathbb{S}^{n-2}$ and we let $o'$ denote the antipodal point to $o$. For $M = \mathbb{S}^n$, we set $\mathcal{E} = \{\rho = \frac{1}{2} \mathrm{diam}(M)\} \cup \{x \in M : \rho(x) < \frac{1}{2} \mathrm{diam}(M) \text{ and } z_x = o'\}$.  Otherwise, we take $\mathcal{E}=\emptyset$. Note in particular that $B^n_R \subset M\setminus \mathcal{E}$ whenever $R< \frac{1}{2}\diam(M)$. Now, define $\mathrm{sign} :  \gamma \setminus \mathcal{E} \to \{-1, 0, 1\}$ by $\mathrm{sign}(z) = 1$ if $z \in \gamma\setminus \mathcal{E}$ lies on the same side of $o$ as $y$, $\mathrm{sign}(o) = 0$, and $\mathrm{sign}(z) = -1$ otherwise.  

We define a function $s: M \setminus \mathcal{E} \to (-\diam(M), \diam(M))$ by setting $s(x) = \mathrm{sign}(z)r(z_x)$. In particular, the Pythagorean theorem applied to the right geodesic triangle $oz_x x$ gives
\begin{equation}
\label{eq:pythag-0}
\begin{cases} \cs(s(x))\cs(\rho(x)) = \cs(r(x)) & M \in \{\mathbb{H}^n, \mathbb{S}^n\} \\ s(x)^2 + \rho(x)^2 = r(x)^2 & M = \mathbb{R}^n \end{cases}, 
\end{equation}
for $x \in B^n_R$. 

The function $s$ is smooth on $M\setminus \mathcal{E}$ and the function $\rho$ is a distance function on $M$. The level sets of $s$ are totally geodesic hypersurfaces in $M$; in particular $s$ is constant on any $k$-dimensional totally geodesic disk in $B^n_R$ which intersects $\gamma$ orthogonally. Indeed, the metric on $M\setminus (\mathcal{E}\cup \gamma) \cong (0, \frac{1}{2}\mathrm{diam}(M)) \times (-\diam(M), \diam(M)) \times \mathbb{S}^{n-2}$ may be written\footnote{Set $\phi = \cs(\rho)$ and $\psi = \sn(\rho)$. Let $\partial_\rho$ and $\partial_s$ denote the coordinate vector fields and let $X$ and $Y$ be orthogonal to $\partial_s$ and $\partial_\rho$. It suffices to compute the sectional curvature of the two-planes generated by $\{\partial_\rho, \partial_s\}$, $\{\partial_\rho, X\}$, $\{\partial_s, Y\}$ and $\{X, Y\}$. These are given, respectively, by $-\frac{\phi''}{\phi}$, $-\frac{\psi''}{\psi}$, $-\frac{\phi'\psi'}{\phi\psi}$, and $\frac{1 - (\phi')^2}{\phi^2}$, each of which is $\curv$.}\begin{equation}\label{eq:g-s-rho}g= d\rho^2 + \cs(\rho)^2 ds^2 + \sn(\rho)^2 g_{\mathbb{S}^{n-2}}\end{equation} and the coordinate vector field $\pr_s = \frac{\pr}{\pr s}$ (suitably extended) generates an isometry of $M$. 

\begin{proposition}\label{Hessian-s-and-rho}
On $M \setminus \mathcal{E}$, we have that $\pr_s = \cs(\rho)^2\nabla s$ is a Killing field. In particular, we have $g(\nabla_X \pr_s, X)=0$ for any vector field $X$. 

On $M\setminus (\mathcal{E}\cup \gamma)$, we further have that $|\nabla \rho|^2=1$ and 
\begin{equation}\label{eq:hess-s}
\nabla^2 s  = \curv\tn(\rho)(d\rho \otimes ds + ds \otimes d\rho),
\end{equation}
\begin{equation}\label{eq:hess-rho}
\nabla^2 \rho =\ct(\rho)(g - d \rho \otimes d \rho -d s \otimes d s).
\end{equation}
\end{proposition} 
\begin{proof}
The presentation of the metric (\ref{eq:g-s-rho}) immediately gives that $\pr_s$ is a Killing field, since the coefficients are independent of $s$. For the Hessians, note indeed $\nabla \rho = \pr_\rho$, so \[\nabla^2 s(X,X) = g(\nabla_X \nabla s, X) = 2\curv \frac{\tn(\rho)}{\cs(\rho)^2} g(X, \pr_\rho) g(\pr_s, X) = 2\curv \tn(\rho) g(X, \nabla \rho) g(X,\nabla s).\] Polarization implies the form (\ref{eq:hess-s}), and one may deduce (\ref{eq:hess-rho}) from the Hessian (\ref{Hessian-of-r}) of $r$ and differentiating the Pythagorean relation (\ref{eq:pythag-0}) twice. 

(We remark that one may instead first compute the Hessian of the distance function $\rho$ using Jacobi fields, and then deduce the results for $s$ from the Pythagorean relation.) 
\end{proof}

\section{Proof of Theorems \ref{hyperbolic-space} and \ref{sphere}}

In this section, we will give the proof of Theorems \ref{hyperbolic-space} and \ref{sphere}. Recall that $\Sigma \subset B^n_R$ is a $k$-dimensional minimal submanifold that passes through a point $y \in B^n_R$ and satisfies $\partial \Sigma \subset \partial B^n_R$. We will assume $d(o, y) > 0$ since the case $o \in \Sigma$ is classical. 

\subsection{The area estimate follows from a good vector field}

The vector field method used by Brendle and Hung requires the construction of a smooth vector field $W=W_y$ on $B^n_R \setminus \{y\}$ with the following three properties:
\begin{enumerate}
\item[(V1)] The divergence estimate: for any $k$-plane $S\subset T_xM$, where $r_y(x)>0$, we have
\[
\div_S W(x) \leq 1,
\]
with equality if and only if $|\nabla^\perp r_y|^2 = 0$ and $|\nabla^\top s|^2 = 0$. 
\item[(V2)] The prescribed residue:
\[
\lim_{x \to y} r_y(x)^{k-1} \langle W(x), \nabla r_y(x) \rangle = - A(\underline{r}(y))
\]
where $\underline{r}(y)$ is given by \eqref{underline-r}. 
\item[(V3)] The boundary vanishing condition: 
\[
W(x) = 0 \;\; \text{if} \;\; x \in \partial B^n_R. 
\]
\end{enumerate}

When $y=o$, one may take $W= W_1 -A(R)W_0$, where $W_0,W_1$ are as in Section \ref{sec:classical}. 
For $d(o,y)> 0$ and $M=\mathbb{R}^n$, Brendle and Hung \cite{BH17} constructed a suitable vector field and verified properties (V1) - (V3) in three short lemmas. 

We will define an analogous vector field $W$ for any space form $M \in \{\mathbb{H}^n, \mathbb{R}^n, \mathbb{S}^n\}$ and any $y\in B^n_R$. This vector field will satisfy properties (V1) - (V3) when $M=\mathbb{H}^n$ and $k \geq 1$, and recovers the vector field of Brendle and Hung when $M=\mathbb{R}^n$ (hence satisfies the same properties in this case, for any $k$). When $M=\mathbb{S}^n$, our vector field $W$ satisfies (V2) and (V3) but only satisfies (V1) under certain restrictions on $R,y$, and $k$. As $y \to o$, the classical vector field $W_1 - A(R)W_0$ naturally arises. 

Once a vector field satisfying properties (V1) - (V3) is found, the proof of the prescribed point area estimate follows from the following proposition. 

\begin{proposition}\label{proof-assuming-vs}
Consider $M \in \{\mathbb{H}^n, \mathbb{R}^n, \mathbb{S}^n\}$. Suppose $\Sigma$ is a $k$-dimensional minimal submanifold in $B^n_R$ with radius $R \in (0, \frac{1}{2}\mathrm{diam}(M))$ which passes through a point $y \in B^n_R$ and satisfies $\partial \Sigma \subset \partial B^n_R$. Suppose that there exists a smooth vector field $W = W_y$ on $B^n_R \setminus \{y\}$ with properties $(\mathrm{V}1)-(\mathrm{V}3)$ above. Then 
\[
|\Sigma| \geq |B^k_{\underline{r}(y)}|. 
\]
Moreover, equality holds if and only if $\Sigma$ is a totally geodesic disk of radius $\underline{r}(y)$ orthogonal to $y$. 
\end{proposition}

\begin{proof}
The following argument works for any $k$. Suppose such a vector field $W$ exists. Consider $\Sigma \setminus B^n_t(y)$ for $t > 0$ small. By the divergence property (V1) and since $\Sigma$ is minimal, we have
\begin{equation}
\label{eq:pf-0}
|\Sigma \setminus B^n_t(y)| = \int_{\Sigma \setminus B^n_t(y)} 1 \geq \int_{\Sigma \setminus B^n_t(y)} \mathrm{div}_{\Sigma}(W) = \int_{\Sigma \setminus B^n_t(y)} \mathrm{div}_{\Sigma}(W^\top) . 
\end{equation}
Now, using that $y$ lies in the interior of $\Sigma$, for $t > 0$ sufficiently small, we have
\[
\partial\big( \Sigma \setminus B^n_t(y)\big) = \partial \Sigma \cup (\Sigma \cap \partial B^n_t(y)). 
\]
Let $\eta_\Sigma$ denote the outward-pointing conormal of $\partial \Sigma$ in $\Sigma$ and let $\nu = \frac{\nabla^\top r_y}{|\nabla^\top r_y|}$ denote the \textit{inward}-pointing conormal of $\Sigma \cap \partial B^n_t(y)$ in $\Sigma$. By the divergence theorem 
\begin{equation}
\label{eq:pf-1}
\int_{\Sigma \setminus B^n_t(y)} \mathrm{div}_{\Sigma}(W^\top) = \int_{\partial \Sigma} \langle W, \eta_\Sigma \rangle - \int_{\Sigma \cap \partial B^n_t(y)} \langle W, \nu \rangle.
\end{equation}
Because $\partial \Sigma \subset \partial B^n_R$, the boundary condition (V3) implies 
\begin{equation}
\label{eq:pf-2}
\int_{\partial \Sigma} \langle W, \eta_\Sigma \rangle = 0.
\end{equation}
On the other hand, because $\Sigma$ is smooth at $y$, as $x \to y$ we have 
\[
\nu(x) = \nabla r_y(x) + o(1)
\]
for $x \in \Sigma$. Therefore, the prescribed residue condition (V2) implies
\[
\lim_{x \to y} r_y(x)^{k-1}  \langle W(x), \nu(x) \rangle =- A(\underline{r}(y)) 
\]
for any $x\in  \Sigma$. On the other hand, since $y \in \Sigma$ the density at $y$ satisfies $\Theta(\Sigma, y) \geq 1$. Since the volume is locally Euclidean, we have
\[
|\Sigma \cap \partial B^n_t(y)| =\Theta(\Sigma, y) |\mathbb{S}^{k-1}| t^{k-1} + o(t^{k-1}) \geq |\mathbb{S}^{k-1}| t^{k-1} + o(t^{k-1})
\]
as $t \to 0$. This implies 
\begin{equation}
\label{eq:pf-3}
-\lim_{t \to 0} \int_{\Sigma \cap \partial B^n_t(y)} \langle W, \nu \rangle = A(\underline{r}(y))|\mathbb{S}^{k-1}| = |B^k_{\underline{r}(y)}|.  
\end{equation}
Combining (\ref{eq:pf-0}), (\ref{eq:pf-1}), (\ref{eq:pf-2}) and (\ref{eq:pf-3}) we conclude that 
\[
|\Sigma| = \lim_{t \to 0} |\Sigma \setminus B^n_t(y)| \geq -\lim_{t \to 0} \int_{\Sigma \cap \partial B^n_t(y)} \langle W, \nu \rangle = |B^k_{\underline{r}(y)}|. 
\]

Finally, if equality holds $|\Sigma| = |B^k_{\underline{r}(y)}|$, then we must have $\mathrm{div}_{\Sigma}(W) \equiv 1$ on $\Sigma \setminus \{y\}$. By property (V1), this means $\nabla^\perp r_y \equiv 0$ and $\nabla^\top s \equiv 0$ on $\Sigma \setminus \{y\}$. These conditions readily imply that $\Sigma$ must be a $k$-dimensional totally geodesic disk orthogonal to $y$. 
\end{proof}

\subsection{Construction of good vector field}
\label{sec:construction}

We have thus reduced the prescribed point problem to finding a vector field satisfying conditions (V1) - (V3) above. Let us write \begin{equation}s_y := s(y) = r(y) > 0.\end{equation} Our ambient vector field $W$ is defined by
\begin{equation}
\label{eq:def-W}
W := \frac{A(r_y) - A(u_s)}{A'(r_y)} \nabla r_y  + \big(B(r_y) - B(u_s)\big)A'(u_s)u_s' \cs(s-s_y)^2 \pr_s.
\end{equation}
where $u_s(x) = u(s(x))$ depends only on $s(x)$, and for notational simplicity we define $u_s'(x) = u'(s(x))$, $u_s''(x) = u''(s(x))$. The functions $B$ and $u$ are defined imminently in Definitions \ref{def:B} and \ref{def:u} respectively. 

\begin{definition}[of $B$]
\label{def:B}
If $k\geq2$ and $(M, k) \neq (\mathbb{R}^n, 2)$, we define a function $B$ by 
\begin{equation}
B(r) = - \int_r^{\frac{1}{4}\mathrm{diam}(M)} \frac{1}{\cs(t)^2 A'(t)}\, dt,
\end{equation}
for $r\in (0, \frac{1}{2}\mathrm{diam}(M))$. If $(M, k) = (\mathbb{R}^n, 2)$, then we instead define $B(r) = \log(r)$. Finally, if $k=1$, then we define $B(r)=\tn(r)$. 

In the above, we have chosen $\frac{1}{4}\mathrm{diam}(M)$ as the upper limit for convenience (so that the integral converges when $M = \mathbb{S}^n$). In practice, we only need the relative values of $B$ and the behaviour near $r=0$. In any case, $B$ is monotone increasing and indeed satisfies 
\begin{equation}
B'(r) = \frac{1}{\cs(r)^2 A'(r)} >0. 
\end{equation}
Note that 
$
B'(r) = r^{1-k} + o(r^{1-k})
$
as $r\to 0$, and hence 
\begin{equation}
B(r) = \begin{cases} -\frac{1}{k-2}r^{2-k} + o(r^{2-k}), & k\neq 2  \\  \log(r) + o(\log(r)) ,& k=2\end{cases}. 
\end{equation}
\end{definition}

\begin{definition}[of $u$]
\label{def:u}
We define $u : [-R, R] \to [0, \mathrm{diam}(M))$ by 
\begin{equation}
\label{eq:def-u}
u(s) := 
\begin{cases}
\cs^{-1}\Big(\frac{\cs(s - s_y)\cs(R)}{\cs(s)}\Big) & M= \{\mathbb{H}^n, \mathbb{S}^n\}\\
\big(R^2 + (s_y-s)^2  - s^2\big)^{\frac{1}{2}} & M = \mathbb{R}^n
\end{cases},
\end{equation}
where again the inverse functions map $(\cdot)^\frac{1}{2} : [0, \infty) \to [0, \infty)$, $\cosh^{-1} : [1, \infty) \to [0, \infty)$, and $\cos^{-1} : (-1, 1] \to [0, \pi)$. We can readily verify that the function $u(s)$ is well-defined for $|s|\leq R$, hence the function $u_s$ is well-defined on $B^n_R$. 

Moreover, by the Pythagorean theorem applied to the right geodesic triangles $oz_x x$ and $yz_x x$, the function $u_s$ satisfies 
\begin{equation}\label{u_s-and-r_y}
\begin{cases}
\cs(u_s) = \cs(r_y) \frac{\cs(R)}{\cs(s)\cs(\rho)}= \cs(r_y) \frac{\cs(R)}{\cs(r)} & M \in \{\mathbb{H}^n, \mathbb{S}^n\}\\
u_s^2  = r_y^2 + R^2-s^2 - \rho^2= r_y^2 + R^2 - r^2 & M = \mathbb{R}^n
\end{cases}.
\end{equation}
\end{definition}

Note that the function $u_s$ has been chosen precisely so that it depends only on $s(x)$, and coincides with $r_y$ on $\pr B^n_R$. Indeed, we have:

\begin{lemma}
\label{lem:prop-u}
The function $u$ satisfies the following properties: 
\begin{enumerate}
\item[(U1)] $u_s(x) = r_y(x)$ if $x \in \partial B^n_R$.
\item[(U2)] $u_s(y) = \underline{r}(y)$. 
\item[(U3)] For any $x \in B^n_R$, if $\cs(u_s(x))\geq 0$, then $u_s(x) \geq r_y(x)$. In particular, the conclusion always holds if $M\in\{\mathbb{H}^n, \mathbb{R}^n\}$.

\end{enumerate}
\end{lemma}

\begin{proof}
Properties (U1) and (U2) are immediate consequences of \eqref{u_s-and-r_y} and \eqref{underline-r} respectively. Property (U3) also follows from \eqref{u_s-and-r_y} and the monotonicity of $\cs$ or $(\cdot)^2$ (recall $r\leq R<\frac{1}{2}\diam M$). 
\end{proof}

\begin{remark}\label{wedge-domain-intro}
When $M = \mathbb{S}^n$, the region $\{\cs(u_s(x)) \geq 0\}$ is precisely $B^n_R \cap B^n_{\frac{\pi}{2}}(y)$. Moreover, when $\cs(u_s) = 0$, the Pythagorean theorem for $\mathbb{S}^n$ implies $\cs(r_y) = 0$. So in addition to property (U1) on the sphere, we also have $u_s(x) = r_y(x)$ for all $x \in B^n_R \cap \partial B^n_{\frac{\pi}{2}}(y)$. This latter property sheds some light on why the ansatz \eqref{eq:def-W} will not always work in the sphere. See Section 5 for more discussion. 
\end{remark}

\begin{remark}
Observe that when $M=\mathbb{R}^n$ and $k \geq 3$, we have $\cs=1$, $A'(r) = r^{k-1}$, $A(r)=\frac{1}{k}r^k$ and $B(r) = -\frac{1}{k-2} r^{2-k}$. On can also write $r_y = |x-y|$, $s_y= |y|$, and $s  = \frac{1}{|y|}\langle x,y\rangle$, so that $\nabla r_y = \frac{x-y}{|x-y|}$ and $\nabla s = \pr_s = \frac{y}{|y|}$. Then $u_s = (R^2 - 2\langle x,y\rangle +|y|^2)^\frac{1}{2}$ and $u_s' = - \frac{|y|}{u_s}$. Making these substitutions in \eqref{eq:def-W} gives exactly the vector field of Brendle and Hung \eqref{brendle-hung-vector} for $k\geq 3$. The reader can check the same is true for $k = 2$. 
\end{remark}

\subsection{Properties of the good vector field}

The boundary-vanishing and prescribed residue of $W$ follow immediately from our chosen ansatz:

\begin{lemma}
\label{lem:boundary}
$W=0$ along the boundary $\partial B^n_R$. 
\end{lemma}
\begin{proof}
This follows immediately from the definition (\ref{eq:def-W}) of $W$ and the property (U1) of the function $u_s$; that is, $u_s=r_y$ on $\pr B^n_R$. 
\end{proof}

\begin{lemma}
\label{lem:residue}
The vector field $W$ satisfies
\[
\lim_{x \to y} r_y(x)^{k-1} \langle W(x), \nabla r_y(x) \rangle = - A(\underline{r}(y))
\]
where $\underline{r}(y)$ is given by \eqref{underline-r}.
\end{lemma}
\begin{proof}
Recall that by property (U2), we have $u_s(y) = \underline{r}(y)$. Recalling \eqref{asymptotics-1} and \eqref{asymptotics-2}, we have $A'(r_y) = r_y^{k-1} + o(r_y^{k-1})$, so that $A(r_y) = \frac{1}{k} r_y^k + o(r_y^k)$ and $B(r_y) = o(r_y^{1-k})$ as $r_y\to 0$. It then follows from (\ref{eq:def-W}) that \[W = - r_y^{1-k} A(\underline{r}(y))\nabla r_y + o(r_y^{1-k}),\] which implies the result. 
\end{proof}

It remains to check when the divergence estimate holds, which will depend on the properties of $u$. We first have the following calculation:

\begin{lemma}
\label{lem:div-W}
Suppose $r_y(x)>0$ and consider a $k$-plane $S\subset T_xM$. Then at $x$ we have 
\begin{equation}
\begin{split}
 \div_S W &=1- \left(1+ k\ct(r_y) \frac{A(u_s) - A(r_y)}{A'(r_y)}\right) |\nabla^\perp r_y|^2 - (u'_s)^2 \frac{\cs(r_y)^2}{\cs(u_s)^2}|\nabla^\top s|^2\\ &\quad - (B(u_s)-B(r_y)) (A'(u_s) u_s' \cs(s-s_y)^2)'\frac{\cs(r_y)^2}{\cs(s-s_y)^2}|\nabla^\top s|^2, 
 \end{split} \end{equation}
where $(A'(u_s) u_s' \cs(s-s_y)^2)'(x) = \left.\frac{d}{ds}\left(A'(u(s)) u'(s) \cs(s-s_y)^2 \right)\right|_{s=s(x)}$. \\

In particular, if 
\begin{enumerate}
\item[(i)] $1+ k\ct(r_y) \frac{A(u_s) - A(r_y)}{A'(r_y)}\geq 0$ on $B^n_R$, and
\item[(ii)] $\frac{\cs(s)^2}{\cs(R)^2}u'(s)^2  + \left(B(u(s))-B(|s-s_y|)\right) \frac{d}{ds}(A'(u(s)) u'(s) \cs(s-s_y)^2) \geq 0$ for all $s\in (-R,R)$, 
\end{enumerate}
then $W$ satisfies (V1).
\end{lemma} 

\begin{proof}
For $x\in B^n_R$, the Pythagorean theorem to the right geodesic triangle $yz_x x$ gives 
\begin{equation}
\label{eq:pythag-y}
\begin{cases} \cs(s-s_y)\cs(\rho) = \cs(r_y) & M \in \{\mathbb{H}^n, \mathbb{S}^n\} \\ (s-s_y)^2 + \rho^2 = r_y^2 & M = \mathbb{R}^n \end{cases}.
\end{equation}
In any case, $\pr_s = \cs(\rho)^2 \nabla s=\frac{\cs(r_y)^2}{\cs(s-s_y)^2} \nabla s$ is Killing and hence satisfies $\div_S \pr_s =0$. Recall the definition (\ref{eq:def-W}) of $W$:
\[
W = \frac{A(r_y) - A(u_s)}{A'(r_y)} \nabla r_y  + \big(B(r_y) - B(u_s)\big)A'(u_s)u_s' \cs(s-s_y)^2 \pr_s.
\]
Taking the divergence along $S\subset T_xM$, we compute 

\begin{align*}
\div_S W &= \div_S \left(\frac{A(r_y)}{A'(r_y)} \nabla r_y \right) - A(u_s) \div_S \left(\frac{1}{A'(r_y)} \nabla r_y\right)
\\& \quad  - \frac{A'(u_s) u'_s}{A'(r_y)} g(\nabla^\top s, \nabla r_y) + B'(r_y) A'(u_s) u'_s \cs(s-s_y)^2 g(\nabla^\top r_y, \pr_s)
\\& \quad - B'(u_s) A'(u_s) (u'_s)^2 \cs(s-s_y)^2 g(\nabla^\top s, \pr_s) 
\\&\quad - (B(u_s)-B(r_y)) (A'(u_s) u_s' \cs(s-s_y)^2)' g(\nabla^\top s, \pr_s).
\end{align*}
Using Proposition \ref{divergences} for the radial fields and simplifying the inner products, we have 

\begin{align*}
\div_S W &=1- \left(1 + k\ct(r_y) \frac{A(u_s) - A(r_y)}{A'(r_y)}\right) |\nabla^\perp r_y|^2
\\&\quad  + \left(- \frac{1}{A'(r_y)}  + B'(r_y)  \cs(r_y)^2 \right) A'(u_s) u'_s g(\nabla^\top r_y, \nabla^\top s)
\\&\quad  - B'(u_s) A'(u_s) (u'_s)^2 \cs(r_y)^2 |\nabla^\top s|^2 
\\&\quad - (B(u_s)-B(r_y)) (A'(u_s) u_s' \cs(s-s_y)^2)' \frac{\cs(r_y)^2}{\cs(s-s_y)^2}|\nabla^\top s|^2. 
\end{align*}
The choice of $B$ implies that the second line vanishes, and the third line simplifies as desired. 

Condition (i) is precisely the statement that the coefficient of $-|\nabla^\perp r_y|^2$ is nonnegative. Similarly, the coefficient of $-\frac{\cs(r_y)^2}{\cs(s-s_y)^2} |\nabla^\top s|^2$ is 
\begin{equation}\label{eq:coeff-s} (u'_s)^2 \frac{\cs(s)^2}{\cs(R)^2}+ (B(u_s)-B(r_y)) (A'(u_s) u_s' \cs(s-s_y)^2)' .\end{equation}
Here we have used that $\cs(u_s) = \frac{\cs(R)\cs(s-s_y)}{\cs(s)}$.

We now note that $r_y$ always lies between $|s-s_y|$ and $u_s$. Indeed, when $M \in \{\mathbb{H}^n, \mathbb{R}^n\}$, the Pythagorean theorem implies $\cosh(s-s_y) \leq \cosh(s-s_y)\cosh(\rho) = \cosh(r_y) \leq \cosh(u_s)$ and $(s-s_y)^2 \leq (s-s_y)^2 + \rho^2 = r_y^2 \leq u_s^2$ respectively, hence $|s-s_y|\leq r_y\leq u_s$. When $M =\mathbb{S}^n$, it is possible for $\cos$ to flip sign and one either has $\frac{\pi}{2} \geq u_s \geq r_y \geq |s -s_y|$ or $\frac{\pi}{2} \leq u_s \leq r_y \leq |s-s_y|$.

Therefore, by monotonicity of $B$, for (\ref{eq:coeff-s}) to be nonnegative on $B^n_R$, we need only check that it is nonnegative when $r_y = u_s$ and when $r_y = |s-s_y|$. The former is obvious, as the second term vanishes and the first term is always nonnegative. Condition (ii) is precisely the resulting ordinary differential inequality arising from the case $r_y=|s-s_y|$. 
\end{proof}

We will later see that first condition, Lemma \ref{lem:div-W}(i), always holds of our choice of $u_s$. To check the second condition, Lemma \ref{lem:div-W}(ii), we begin with the following calculation:

\begin{lemma}
\label{lem:F'}
Let \[
F(s) = A'(u(s)) u'(s)\cs(s-s_y)^2. 
\]

For $M\in \{\mathbb{H}^n, \mathbb{R}^n, \mathbb{S}^n\}$, we have  
\begin{equation}\label{eq-F'}
F'(s) =  \frac{\sn(u(s))^{k-4} \cs(u(s))\sn(s_y)^2}{\cs(s)^2} (k\cs(u(s))^2 - 2).
\end{equation} 
\end{lemma}

\begin{proof}
First suppose $M = \mathbb{R}^n$. Then 
\[
F(s) = u(s)^{k-1} u'(s). 
\]
and so \begin{align}\label{eq:F'-Rn} 
\nonumber F'(s) &= u(s)^{k-2}\left( (k-1)u'(s)^2 + u(s) u''(s)\right) \\
& = u(s)^{k-2}\left( (k-2) u'(s)^2 + \frac{1}{2}(u^2)''(s)\right)\\ 
\nonumber & = u(s)^{k-4} s_y^2(k-2), 
\end{align}
where we have used that $u^2(s) = R^2 + s_y^2 - 2ss_y$ implies $(u^2)''(s)=0$ and $u'(s) = \frac{(u^2)'(s)}{2u(s)} =  \frac{-s_y}{u(s)}$. Evidently, this gives \eqref{eq-F'} for $M = \mathbb{R}^n$. 

Next, suppose $M \in \{ \mathbb{H}^n, \mathbb{S}^n\}$. Then
\[
F(s) = \sn(u(s))^{k-1} u'(s) \cs(s -s_y)^2,
\]
so that 
\[\begin{split}
\frac{F'(s)}{\sn(u(s))^{k-2}} =& (k-1) \cs(u(s)) u'(s)^2 \cs(s-s_y)^2 \\&+ \sn(u(s)) u''(s) \cs(s-s_y)^2 + \sn(u(s)) u'(s) \frac{d}{ds}(\cs(s-s_y)^2).
\end{split}\]
Note that $\frac{d}{ds}\cs(u(s)) = \curv \sn(u(s)) u'(s) $ and $\frac{d^2}{ds^2}\cs(u(s)) = -\curv \left( \cs(u(s))u'(s)^2 + \sn(u(s))u''(s)\right)$. Therefore 

\begin{equation}
\label{eq:F'-curv}
\frac{F'(s)}{\sn(u(s))^{k-2}} = (k-2)  \cs(u(s)) u'(s)^2 \cs(s-s_y)^2 - \curv \frac{d}{ds} \left( \cs(s-s_y)^2 \frac{d}{ds}\cs(u(s))\right) .
\end{equation}
Recall that 
\[
\cs(u(s)) = \cs(R) \frac{\cs(s-s_y)}{\cs(s)}.
\]
Differentiating this, we find that
\begin{equation}\label{derivative-cos-u} \frac{d}{ds} \cs(u(s)) = -\curv \sn(u(s)) u'(s) = \curv \frac{\cs(R) \sn(s_y)}{ \cs(s)^2}. \end{equation}
Differentiating once more gives 
\begin{align}
\label{eq:ddcsu}
\frac{d}{ds} \left( \cs(s-s_y)^2 \frac{d}{ds}\cs(u(s))\right)  &=\frac{d}{ds}\left(\cs(s-s_y)^2  \curv \frac{\cs(R) \sn(s_y)}{ \cs(s)^2}\right) \\
\nonumber &= 2\curv^2 \cs(R) \sn(s_y)^2 \frac{\cs(s-s_y)}{\cs(s)^3} \\
\nonumber & = 2\frac{\sn(s_y)^2}{\cs(s)^2} \cs(u(s)).  
\end{align} 
On the other hand squaring (\ref{derivative-cos-u}) gives 
\begin{equation} \label{eq:u'sq} u'(s)^2 = \frac{\cs(R)^2 \sn(s_y)^2}{\sn(u(s))^2 \cs(s)^4} , \end{equation}
and using this and (\ref{eq:ddcsu}) in (\ref{eq:F'-curv}) gives 
\[
\frac{F'(s)}{\sn(u(s))^{k-2} \cs(u(s))} = (k-2)\frac{\cs(R)^2 \sn(s_y)^2 \cs(s-s_y)^2}{\sn(u(s))^2 \cs(s)^4} - 2\curv \frac{\sn(s_y)^2}{\cs(s)^2} .
\]
Multiplying some factors over gives
\[
\frac{F'(s) \cs(s)^2}{\sn(u(s))^{k-4} \cos(u(s))\sn(s_y)^2} = (k-2)\frac{\cs(R)^2 \cs(s-s_y)^2}{\cs(s)^2} - 2 \curv \sn(u(s))^2. 
\]
Finally, using $\cs^2+\curv \sn^2=1$ and the definition of $u(s)$ gives
\[
\frac{F'(s) \cs(s)^2}{\sn(u(s))^{k-4} \cs(u(s))\sn(s_y)^2} =  k\cs(u(s))^2 - 2.
\]
\end{proof}

We deduce that if $M \in \{\mathbb{H}^n, \mathbb{R}^n\}$, then the inequality in Lemma \ref{lem:div-W}(ii) is satisfied for any $k\geq 1$:

\begin{lemma}\label{lem:divergence}
Suppose that $M\in \{\mathbb{R}^n,\mathbb{H}^n\}$ and $k\geq 1$. Then $W$ satisfies property (V1). 
 \end{lemma}

\begin{proof}
By Lemma \ref{lem:prop-u}, in each of the stated cases, we have $u_s \geq r_y$, hence $A(u_s)\geq A(r_y)$ and $B(u_s)\geq B(r_y)$. In particular, Lemma \ref{lem:div-W}(i) is satisfied as $A(u_s)-A(r_y)\geq 0$. Thus it suffices to check Lemma \ref{lem:div-W}(ii):
\begin{equation}
\label{eq:nonnegative-F'-1}
\frac{\cs(s)^2}{\cs(R)^2}(u'_s)^2 + (B(u_s)-B(r_y)) F'(s)\geq 0. 
\end{equation}

First, we consider $k\geq 2$. If $M\in \{\mathbb{H}^n, \mathbb{R}^n\}$, then by Lemma \ref{lem:F'} we see that already $F'(s)\geq 0$ (note that $\cosh\geq 1$), which implies the desired inequality. 

Now consider $k=1$ (and again $\curv \in\{-1,0\}$). Then by definition $B= \tn$, so by (\ref{eq:u'sq}), the left hand side of Lemma \ref{lem:div-W}(ii) becomes

\[\frac{\cs(s)^2}{\cs(R)^2}(u'_s)^2 + (B(u_s)-B(r_y)) F'(s) = \frac{\sn(s_y)^2}{\sn(u_s)^2 \cs(s)^2}\left( 1+ \left(1-\frac{\tn(r_y)}{\tn(u_s)} \right)( \cs(u_s)^2-2)\right).
\]
But now note that $\cs^2 -2 = -\curv \sn^2-1$, so 
\[
\begin{split}
1+ \left(1-\frac{\tn(r_y)}{\tn(u_s)} \right)( \cs(u_s)^2-2) &=1- \left(1-\frac{\tn(r_y)}{\tn(u_s)} \right)( \curv \sn(u_s)^2 +1) 
\\&= -\curv \left(1-\frac{\tn(r_y)}{\tn(u_s)} \right)\sn(u_s)^2 + \frac{\tn(r_y)}{\tn(u_s)}.
\end{split}
\]
As $\curv \leq 0$ and we have $0\leq \tn(r_y) \leq \tn(u_s)$, the right hand side is nonnegative as desired. 
\end{proof}

If $M=\mathbb{S}^n$, then we find that condition Lemma \ref{lem:div-W}(ii) can only be satisfied (for all $|s|\leq R$) when $k>2$. 

In what follows, we simplify the condition Lemma \ref{lem:div-W}(ii) and also give an explicit sufficient condition in terms of $k, s_y, R$. 

\begin{lemma}
\label{lem:sphere}
Consider $M=\mathbb{S}^n$ and recall that the function $u: [-R,R] \to (0,\pi)$ is defined by $u(s)=\cos^{-1}\Big(\frac{\cos(s - s_y)\cos(R)}{\cos(s)}\Big)$. Then 
\begin{equation} \label{eq:u'} 
u'(s) = -\frac{\cos(R) \sin(s_y)}{\sin(u(s)) \cos(s)^2} < 0. 
\end{equation}
In particular, $\cos (u(s)) \geq  \cos(R+s_y)$ for all $s\in[-R,R]$.
\end{lemma}

\begin{proof}
Recall from (\ref{derivative-cos-u}) that \[\frac{d}{ds} \cos(u(s)) = - \sin(u(s)) u'(s) =\frac{\cos(R) \sin(s_y)}{ \cos(s)^2}.\] 
 The first statement follows since by definition $u(s) \in (0,\pi)$ and hence $\sin(u(s))>0$. Thus $u$ is monotone decreasing, so since $\cos$ is also decreasing we have 
 \[
 \cos(u(s)) \geq \cos(u(-R)) = \cos(R+s_y).
 \] 
\end{proof}

\begin{lemma}
\label{lem:divergence-sph}
Suppose $M=\mathbb{S}^n$ and that 
\begin{equation}
\label{eq:cond-sphere-0}
1  + \left(B(u(s))-B(|s-s_y|)\right) \sin(u(s))^{k-2} \cos(u(s))(k\cos(u(s))^2 - 2) \geq 0
\end{equation}
for all $s\in (-R,R)$. Then $W$ satisfies (V1). 
Moreover, condition (\ref{eq:cond-sphere-0}) holds so long as
\begin{equation}
\label{eq:cond-sphere}
 \cos(s_y + R) \geq \sqrt{\frac{2}{k}}.
\end{equation}
\end{lemma}
\begin{proof}
For $M = \mathbb{S}^n$, we remarked after Proposition \ref{divergences} that $1 + k \cot(r_y) \frac{A(u_s) - A(r_y)}{A'(r_y)} \geq 0$ always holds. Indeed, on the sphere one already has $1 - k \cot(r_y) \frac{A(r_y)}{A'(r_y)} \geq 0$. So when $\cos(r_y) \geq 0$, the inequality is clear. When $\cos(r_y) \leq 0$, then $\frac{\pi}{2} \leq u_s \leq r_y$. Hence $A(u_s) \leq A(r_y)$ and again the desired inequality holds. Therefore $W$ satisfies (V1) so long as Lemma \ref{lem:div-W}(ii) is satisfied. Using Lemma \ref{lem:F'} \eqref{eq-F'} and Lemma \ref{lem:sphere} \eqref{eq:u'}, the condition Lemma \ref{lem:div-W}(ii) 
\[
\frac{\cos(s)^2}{\cos(R)^2}u'(s)^2  + \left(B(u(s))-B(|s-s_y|)\right) \frac{d}{ds}(A'(u(s)) u'(s) \cos(s-s_y)^2) \geq 0
\]
becomes
\[
\frac{ \sin(s_y)^2}{\sin(u(s))^2 \cos(s)^2}  + \left(B(u(s))-B(|s-s_y|)\right)\frac{\sin(u(s))^{k-4} \cos(u(s))\sin(s_y)^2}{\cos(s)^2} (k\cos(u(s))^2 - 2) \geq 0. 
\]
Multiplying by $\frac{\sin(u(s))^2 \cos(s)^2}{\sin(s_y)^2}$ gives (\ref{eq:cond-sphere-0}).

Now suppose that $\cos(s_y+R) \geq \sqrt{\frac{2}{k}}$. Then by Lemma \ref{lem:sphere} we have $\cos(u(s))>\sqrt{\frac{2}{k}}>0$. So by property (U3) we have $u_s\geq r_y$ and hence $B(u_s)\geq B(r_y)$. Moreover, we have $\cos(u(s))^2 \geq \frac{2}{k}$ and so by Lemma \ref{lem:F'} we have \[F'(s)=\frac{d}{ds}(A'(u(s)) u'(s) \cos(s-s_y)^2)\geq 0.\] Together these imply (\ref{eq:cond-sphere-0}), which completes the proof. 
\end{proof}

\subsection{Prescribed point area estimates}

\begin{proof}[Proof of Theorems \ref{Brendle-Hung}, \ref{hyperbolic-space} and \ref{sphere}]
Lemma \ref{lem:divergence}, resp. Lemma \ref{lem:divergence-sph}, establishes that our vector field $W$ defined in (\ref{eq:def-W}) satisfies property (V1). Lemmas \ref{lem:residue} and \ref{lem:boundary} ensure that $W$ satisfies properties (V2) and (V3) respectively. The desired area estimates then follow immediately from Proposition \ref{proof-assuming-vs}. The only outstanding case is when $(M, k) = (\mathbb{S}^n, 1)$ which is addressed in Proposition \ref{prop:geodesic-estimate} below. 
\end{proof}

Having established area estimates via our vector field $W$, we briefly reflect on a similar ansatz:

\begin{remark}
\label{rmk:super}
Consider
\begin{equation}
\label{eq:tilde-W}
\tilde{W} = \frac{A(r_y) - A(u)}{A'(r_y)} \nabla r_y  + (\tilde{B}(r_y)-\tilde{B}(u))\nabla u,
\end{equation}
where $\tilde{B}' = \frac{1}{A'}$ and $u$ is a function on $B^n_R$ with $(u-r_y)|_{\pr B^n_R}=0$. With this $\tilde{B}$, it follows that for any $k$-plane $S\subset T_xM$, 
\begin{equation}
\div_S \tilde{W} = 1- \left(1+ k\ct(r_y) \frac{A(u) - A(r_y)}{A'(r_y)}\right) |\nabla^\perp r_y|^2 - |\nabla u|^2 - (B(u)-B(r_y)) \tr_S \nabla^2 A(u). 
\end{equation}

To ensure $\div_S\tilde{W} \leq 1$, it is thus sufficient that $u\geq r_y$ and $\tr_S \nabla^2 A(u) \geq 0$. This was previously observed by Berndtsson \cite{Ber19} in the Euclidean setting. Moreover, if (V1) does hold, then it follows that any $k$-dimensional minimal submanifold $\Sigma$ in $B^n_R$ satisfies $|\Sigma|\geq A(u(y))$. When $M=\mathbb{R}^n$ and $u=u(s)$ as above, $\tilde{W}$ again coincides with the vector field used by Brendle and Hung. 

However, we emphasise that our vector field $W$, defined in (\ref{eq:def-W}) with $u=u(s)$, is not of the form (\ref{eq:tilde-W}) when $\curv\neq 0$. Indeed, the second `correction' term in our ansatz is built around the fact that $\nabla \pr_s=0$; on the other hand, when $\curv \neq 0$, $\nabla u_s = u'_s \nabla s$ introduces cross terms which are not simple to control. 
\end{remark}

\section{Geodesics in $\mathbb{S}^n$}
\label{sec:geodesics}

For geodesics ($k=1$), it remains to prove the prescribed point area estimate in $M=\mathbb{S}^n$. 

\begin{proposition}\label{prop:geodesic-estimate}
Let $M = \mathbb{S}^n$, $R \in (0, \frac{\pi}{2})$, and suppose $\sigma$ is a geodesic segment in $B^n_R$ which passes through a point $y \in B^n_R$ and satisfies $\partial \sigma \subset \partial B^n_R$. Then 
\begin{equation}\label{geodesic-estimate}
|\sigma| \geq 2\, \underline{r}(y). 
\end{equation}
with equality if and only if $\sigma$ intersects $\gamma$ orthogonally at $y$, where $\gamma$ is the unique maximal geodesic through $o$ and $y$. 
\end{proposition}

\begin{proof}
Given any nonzero tangent vector $X$ at $y$, there is a unique maximal geodesic $\tilde{\sigma}$ which is tangent to $X$ at $y$. The length of $\tilde{\sigma} \cap B_R$ is determined by the angle $\alpha \in [0, \pi)$ that $\tilde{\sigma}$ makes with $\gamma$ at $y$ (i.e. $\cos(\alpha) = - g(X, \gamma'(y))$). 

Let $\tilde{\sigma}_\alpha$ be any geodesic segment with one endpoint at $y$ and the other endpoint $z \in \pr B^n_R$ so that $\angle oyz = \alpha$. Set $l(\alpha) = |\tilde{\sigma}_\alpha|$. Then $|\tilde{\sigma} \cap B_R| = l(\alpha) + l(\pi-\alpha)$. Let $l_{\min} := \min_{\alpha} (l(\alpha) + l(\pi-\alpha))$. 
By the spherical law of cosines applied to the geodesic triangle $oyz$, we have 
\[
\cos(R) = \cos(s_y) \cos(l(\alpha)) +\sin(s_y) \sin(l(\alpha))  \cos(\alpha).
\]
In particular, notice that $\cos(\alpha) \tan(s_y) =  \frac{C}{\sin (l(\alpha))}- \cot (l(\alpha)) $, where $C=\cos(\underline{r}(y)) \geq 0$. Therefore, to find the geodesic of shortest length, we are looking to optimise $l_1 + l_2$, subject to the smoothness constraint at $y$ - i.e. $\cos(\pi - \alpha) = - \cos(\alpha)$, which implies:
\begin{equation}
\label{eq:constraint}
\cot(l_1) - \frac{C}{\sin(l_1)} = - \cot(l_2) + \frac{C}{\sin(l_2)} = \cos(\alpha)\tan(s_y). 
\end{equation} 

Squaring (\ref{eq:constraint}) gives 
\begin{equation}\label{squared-constraint}
\frac{\cos(l_1)^2 -2C \cos(l_1) + C^2 }{\sin(l_1)^2} = \frac{\cos(l_2)^2 -2C \cos(l_2) + C^2 }{\sin(l_2)^2} \leq \tan(s_y)^2. 
\end{equation}
Note that the only way this inequality becomes an equality is when $\cos^2\alpha=1$; this corresponds to the geodesic $\gamma \cap B_R$, which has length $2R$. 

On the one hand, where strict inequality holds in (\ref{squared-constraint}), by the method of Lagrange multipliers, any extreme values occur when 
\begin{equation}
\label{eq:first-order-constraint}
  \frac{C\cos(l_1)-1}{\sin(l_1)^2} = \lambda = \frac{C\cos(l_2)-1}{\sin(l_2)^2}.\end{equation}
This corresponds to the first order condition $l'(\alpha) = l'(\pi -\alpha)$. 

On the other hand, using $\cos^2 + \sin^2 =1$ on \eqref{squared-constraint} we get
 \begin{equation}
 \label{eq:sq-constraint-2}
 \frac{1 -2C \cos(l_1) + C^2 }{\sin(l_1)^2} = \frac{1 -2C \cos (l_2 )+ C^2}{\sin(l_2)^2}.
 \end{equation} 
Adding twice (\ref{eq:first-order-constraint}) to (\ref{eq:sq-constraint-2}) gives
\[
\frac{(1-C)^2}{\sin(l_1)^2} = \frac{(1-C)^2}{\sin(l_2)^2}.
\] 
We conclude that $l_1=l_2 = \frac{1}{2}l_{\min}$. By (\ref{eq:constraint}) we see that $\cos(\alpha)=0$, hence $\cos(\frac{1}{2} l_{\min}) = C = \cos(\underline{r}(y))$. This completes the proof. 
\end{proof}

\section{Rotationally symmetric domains}
\label{sec:other-domains}

In this section, we discuss how our ansatz may be applied to certain domains which are rotationally symmetric about the geodesic $\gamma$ (which connects $o$ to $y$). We also investigate the prescribed point problem from the perspective of finding the smallest domain on which the desired area estimate holds. In particular, following exactly the same proof as Proposition \ref{proof-assuming-vs}, we have:

\begin{proposition}\label{proof-assuming-vs-domain}
Consider $M \in \{\mathbb{H}^n, \mathbb{R}^n, \mathbb{S}^n\}$ and let $\Omega$ be a domain in $M$ containing $y$. Suppose $\Sigma$ is a $k$-dimensional minimal submanifold in $\Omega$ which passes through $y$ and satisfies $\partial \Sigma \subset \partial \Omega$. Suppose that there exists a smooth vector field $W = W_y$ on $\Omega \setminus \{y\}$ which satisfies properties (V1) and (V2), as well as the boundary vanishing property
\begin{enumerate}
 \item[(V3')]  $W(x) = 0$ for all $x \in \partial \Omega$. 
 \end{enumerate}

Then 
\begin{equation}
\label{eq:area-est-dom}
|\Sigma| \geq |B^k_{\underline{r}(y)}|. 
\end{equation}

\end{proposition}

\begin{definition}
\label{def:sym-dom}
We say that a domain $\Omega$ is a rotationally symmetric graph over $\gamma$ if, for each $s_0$, the intersection $\Omega\cap \{s=s_0\}$ is a totally geodesic $(n-1)$-disk of radius $R(s_0) < \frac{1}{2}\diam(M)$. 
\end{definition}

For the estimate of Proposition \ref{proof-assuming-vs-domain} to possibly hold on a rotationally symmetric domain, one needs $R(s_y) \geq \underline{r}(y)$. For the estimate to be sharp, we may also assume that $R(s_y) = \underline{r}(y)$. For example, by the Pythagorean theorem applied to $oz_x x$, the ball $\Omega=B^n_R$ may be described by $R(s) = \bar{R}(s)$, where 

\begin{equation}
\label{eq:Rs-ball}
\begin{cases}
\cs(\bar{R}(s)) = \frac{\cs(R)}{\cs(s)} & M \in \{\mathbb{H}^n, \mathbb{S}^n\}\\
\bar{R}(s)^2 = R^2 - s^2 & M = \mathbb{R}^n
\end{cases}.
\end{equation}

Note that, in particular, if a suitable vector field exists on $\Omega\subset B^n_R$, then the area estimate may be deduced trivially for any minimal submanifold $\Sigma \subset B^n_R$ (since it holds for $\Sigma\cap \Omega$). 

We construct $W$ in exactly the same way as before, namely

\begin{equation}
\label{eq:def-W-domain}
W := \frac{A(r_y) - A(u_s)}{A'(r_y)} \nabla r_y  + \big(B(r_y) - B(u_s)\big)A'(u_s)u_s' \cs(s-s_y)^2 \pr_s,
\end{equation}

where now we define

\begin{equation}
\label{eq:def-u-domain}
u(s) := 
\begin{cases}
\cs^{-1}\Big(\cs(s-s_y) \cs(R(s))\Big) & M= \{\mathbb{H}^n, \mathbb{S}^n\}\\
\big(R(s)^2 + (s-s_y)^2 \big)^{\frac{1}{2}} & M = \mathbb{R}^n
\end{cases}.
\end{equation}
The Pythagorean theorem applied to the right geodesic triangles $oz_x x$ and $yz_x x$ now gives
\begin{equation}
\begin{cases}
\cs(u_s)= \cs(r_y) \frac{\cs(R(s))}{\cs(\rho)} & M \in \{\mathbb{H}^n, \mathbb{S}^n\}\\
u_s^2 = r_y^2 + R(s)^2 -\rho^2 & M = \mathbb{R}^n
\end{cases}.
\end{equation}
When $x\in \pr\Omega$, then we have $\rho = R(s)$, so in particular $u_s =r_y$ and hence $W=0$. So property (V3') is satisfied. Moreover, property (V2) follows exactly as in Lemma \ref{lem:residue}. 

\subsection{Intersection with a hemisphere}
\label{sec:wedge}

Let $M=\mathbb{S}^n$, $R\in (0,\frac{\pi}{2})$, $y\in B^n_R$ and consider the `wedge' domain $\Omega_w = B^n_R \cap B^n_{\frac{\pi}{2}}(y)$. If $s_y+R > \frac{\pi}{2}$, then the boundary $\pr \Omega_w$ consists of $\{s= s_y-\frac{\pi}{2}\} \cap B^n_R$ and $\pr B^n_R \cap B^n_{\frac{\pi}{2}}(y)$. This is an example of a rotationally symmetric graph over $\gamma$, with 
\[
R(s) =\bar{R}(s) = \cos^{-1}\left(\frac{\cos(R)}{\cos(s)}\right)
\]
for $s \in (s_y - \frac{\pi}{2}, R)$. Moreover, on $\{s= s_y-\frac{\pi}{2}\} \cap B^n_R$ we have $u_s = r_y =\frac{\pi}{2}$. Thus, as we pointed out in Remark \ref{wedge-domain-intro}, the vector field from \eqref{eq:def-W} satisfies properties (V2) and (V3') with respect to $\Omega_w$. 

So by the proof of Lemma \ref{lem:divergence-sph}, for $W$ to satisfy (V1) on $\Omega_w$, it suffices that condition (\ref{eq:cond-sphere-0}),
\begin{equation}\label{eq:cond-sphere-00}
1  + \left(B(u(s))-B(|s-s_y|)\right) \sin(u(s))^{k-2} \cos(u(s))(k\cos(u(s))^2 - 2) \geq 0,
\end{equation}
 holds for $s \in (s_y - \frac{\pi}{2}, R)$. In particular, it actually suffices to check (\ref{eq:cond-sphere-0}) only on this smaller range of $s$. 

However, this leads to an apparent obstruction to the success of our ansatz (\ref{eq:def-W-domain}) in the sphere: Whenever $R$, $k$, and $s_y$ are such that the vector field satisfies (V1), then Proposition \ref{proof-assuming-vs-domain} proves that a totally geodesic disk $B^k_{\underline{r}(y)}$ (orthogonal to $\gamma$ at $y$) has least area among $k$-dimensional minimal submanifolds in $\Omega_w$ passing through $y$. But in the wedge, we can consider the totally geodesic $k$-dimensional submanifolds which instead \textit{contain} the geodesic $\gamma$, and they intersect $\Omega_w$ in totally geodesic $k$-disks of radius $\overline{r}(y):=\frac{1}{2}(R + \frac{\pi}{2} -s_y)$. When $s_y +R > \frac{\pi}{2}$, it is straightforward to find $R, s_y$ such that $\overline{r}(y) < \underline{r}(y) = \cos^{-1}\left( \frac{\cos(R)}{\cos(s_y)}\right)$. That is, the totally geodesic disk parallel to $\gamma$ intersects $\Omega_w$ with \textit{less} area than the orthogonal disk. In these cases, the area estimate (\ref{eq:area-est-dom}) evidently cannot hold in $\Omega_w$, and so our ansatz cannot be used to prove the estimate in $\Omega_w$, nor in $B^n_R$. We nevertheless conjecture that (\ref{eq:area-est-dom}) does hold in $B^n_R$, but a different ansatz or method would be needed to prove it. 

\begin{figure}[h]
\centering
\includegraphics[width=0.5\textwidth]{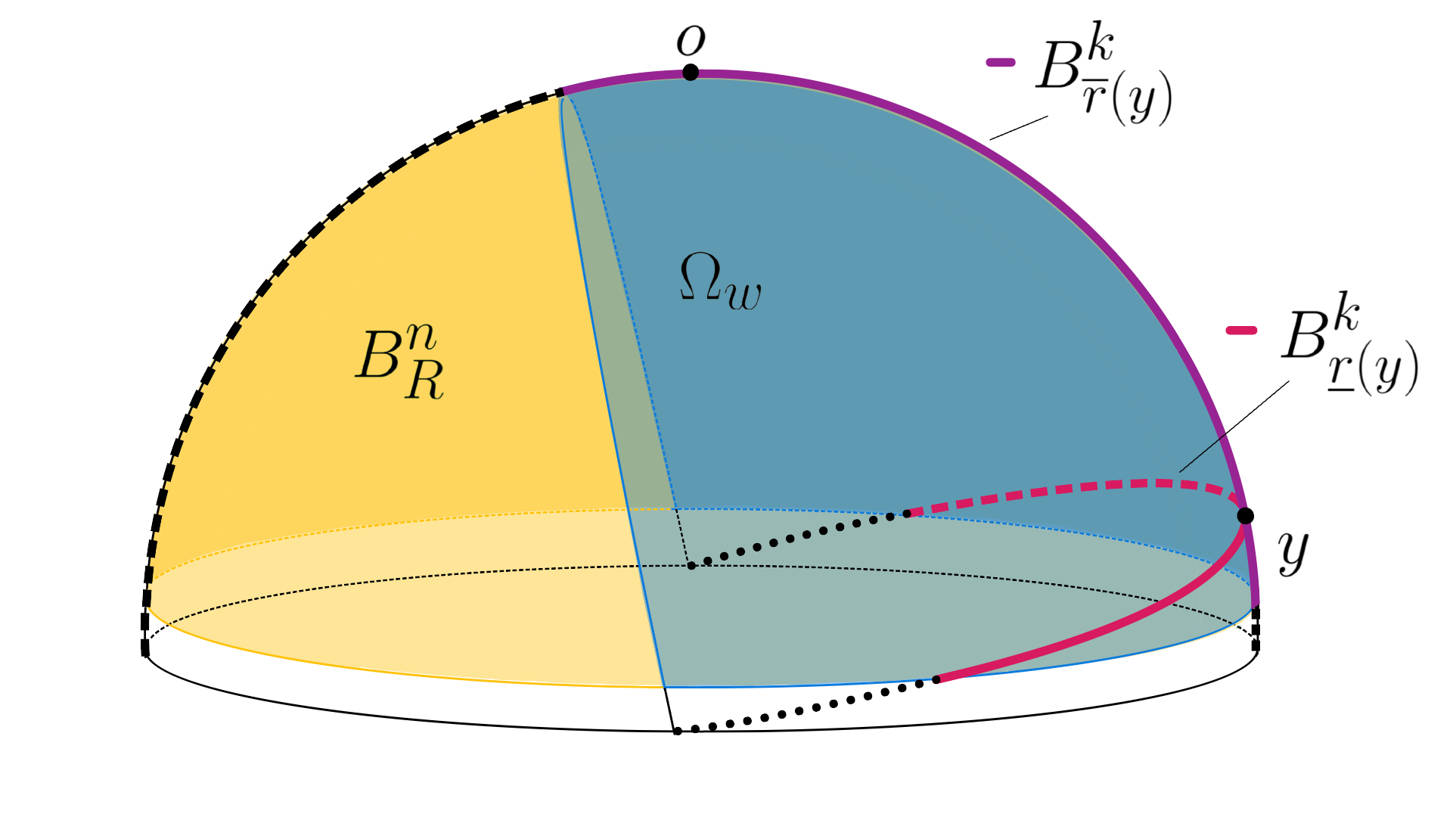}
\caption{Representation of the wedge domain $\Omega_w$, the totally geodesic disk $B^k_{\overline{r}(y)}$ parallel to the geodesic $\gamma$ connecting $o$ to $y$, and the orthogonal disk $B^k_{\underline{r}(y)}$.}
\label{fig:s-and-rho}
\end{figure}

We note that if $\eqref{eq:area-est-dom}$ holds in $B^n_R$ for every $R \in (0, \frac{\pi}{2})$ in $\mathbb{S}^n$, then since $\underline{r}(y) \to \frac{\pi}{2}$ as $R \to \frac{\pi}{2}$, it would imply the following:

\begin{conjecture}
Suppose $\Sigma$ is a $k$-dimensional minimal submanifold in the hemisphere $\mathbb{S}^n_+$  which satisfies $\partial \Sigma \subset \partial \mathbb{S}^n_+$. Then 
\begin{equation}
|\Sigma| \geq |\mathbb{S}^k_+| = \frac{1}{2} |\mathbb{S}^k|. 
\end{equation}
\end{conjecture}
This conjecture is easy to verify when $\Sigma$ is a \textit{free-boundary} minimal submanifold in $\mathbb{S}^n_+$. Also, if $\Sigma$ is a \textit{embedded} two-sided minimal hypersurface, then the following sketch supports the conjecture: Let $\nu$ be the unit normal and notice that $(\cos r)\nu$ generates a variation of $\Sigma$ which fixes the boundary (!). Similar to Case 3 of \cite[Corollary 5.2]{Zhu22}, one has $L_\Sigma (\cos r) = |A|^2 \cos r$, where $L_\Sigma$ is the Jacobi operator. It follows that, unless $\Sigma$ is totally geodesic, the second variation of area is strictly negative; in particular, $\Sigma$ is unstable for variations fixing its boundary. Then $\Sigma$ is a good barrier for solving the Plateau problem in each side of $\mathbb{S}^n_+ \setminus \Sigma = \Omega_+ \cup \Omega_-$ with boundary $\pr \Sigma$. The result is either an area-minimising hypersurface $\tilde{\Sigma}$, or one of the boundaries $E_{\pm} := \partial \Omega_{\pm} \setminus \Sigma$ of $\pr \mathbb{S}^n_+\setminus \Sigma$. If the former case is produced, the argument above shows that any area-minimising hypersurface with boundary in $\mathbb{S}^n_+$ must be totally geodesic, hence $|\Sigma| \geq |\tilde{\Sigma}| = \frac{1}{2} |\mathbb{S}^{n-1}_+|$. Otherwise, the Plateau problem detects the equator on both sides, so we would have $|\Sigma| \geq \max(|E_+|, |E_-|) \geq \frac{1}{2} |\mathbb{S}^{n-1}|$, since $|E_+|+|E_-| = |\pr\mathbb{S}^n_+| = |\mathbb{S}^{n-1}|.$

\subsection{Optimal domains}

Following the proofs of Lemmas \ref{lem:divergence} and \ref{lem:divergence-sph}, we observe that:

\begin{lemma}
Let $\Omega$ be a rotationally symmetric graph over $\gamma$ so that $R(s)$ is defined over $(a,b)$, and suppose that
\begin{equation}
\label{eq:ODI-domain}
\frac{\cs(s)^2}{\cs(R)^2}u'(s)^2  + \left(B(u(s))-B(|s-s_y|)\right) \frac{d}{ds}(A'(u(s)) u'(s) \cs(s-s_y)^2) \geq 0
\end{equation}
 for all $s \in (a,b)$. 
 
Then $W$ satisfies property (V1). 
 \end{lemma}

Given $y$, we define an `optimal' domain $\Omega$ to be the smallest rotationally symmetric domain containing $y$ which satisfies (\ref{eq:ODI-domain}), regarded as an ordinary differential inequality for $R(s)$. 

We conclude with the following observations:

\begin{remark}
By Proposition \ref{proof-assuming-vs-domain} and the discussion above, the area estimate $|\Sigma| \geq |B^k_{\underline{r}(y)}|$ holds for minimal submanifolds $\Sigma$ through $y\in \Omega$, where the domain $\Omega$ has profile $R(s)$ satisfying (\ref{eq:ODI-domain}). 

Our method for the prescribed point area estimate in the ball $B^n_R$ relied essentially on verifying when $\bar{R}$ is a subsolution for (\ref{eq:ODI-domain}). In this sense, the success of our method corresponds to whether the optimal domain $\Omega$ is contained in $B^n_R$. (If so, one may trivially deduce the estimate for $\Sigma \subset B^n_R$ by applying the estimate for $\Sigma \cap \Omega$.) Indeed, in all cases where we have proven that the estimate holds in the ball, it appears that $B^n_R$ is \textit{not} the optimal domain - that is, we can exhibit a strictly smaller domain $\Omega$ on which $u$ satisfies (\ref{eq:ODI-domain}). 
\end{remark}

\bibliographystyle{amsalpha}
\bibliography{prescribed-point-area-estimate-v10}

\providecommand{\bysame}{\leavevmode\hbox to3em{\hrulefill}\thinspace}
\providecommand{\MR}{\relax\ifhmode\unskip\space\fi MR }
\providecommand{\MRhref}[2]{%
  \href{http://www.ams.org/mathscinet-getitem?mr=#1}{#2}
}
\providecommand{\href}[2]{#2}
\begin{thebibliography}{AHO74}

\bibitem[AHO74]{AHO74}
H.~Alexander, D.~Hoffman, and R.~Osserman, \emph{Area estimates for
  submanifolds of {E}uclidean space}, Symposia {M}athematica, {V}ol. {XIV}
  ({C}onvegno di {T}eoria {G}eometrica dell'{I}ntegrazione e {V}ariet\`a
  {M}inimali, {INDAM}, {R}ome, 1973), 1974, pp.~445--455. \MR{0388253}

\bibitem[And82]{And82}
Michael~T. Anderson, \emph{Complete minimal varieties in hyperbolic space},
  Invent. Math. \textbf{69} (1982), no.~3, 477--494. \MR{679768}

\bibitem[AO75]{AO75}
H.~Alexander and R.~Osserman, \emph{Area bounds for various classes of
  surfaces}, Amer. J. Math. \textbf{97} (1975), no.~3, 753--769. \MR{380596}

\bibitem[Ber19]{Ber19}
Bo~Berndtsson, \emph{Superforms, supercurrents, minimal manifolds and
  {R}iemannian geometry}, Arnold Math. J. \textbf{5} (2019), no.~4, 501--532.
  \MR{4068873}

\bibitem[BH17]{BH17}
Simon Brendle and Pei-Ken Hung, \emph{Area bounds for minimal surfaces that
  pass through a prescribed point in a ball}, Geom. Funct. Anal. \textbf{27}
  (2017), no.~2, 235--239. \MR{3626612}

\bibitem[Bre12]{Br12}
Simon Brendle, \emph{A sharp bound for the area of minimal surfaces in the unit
  ball}, Geom. Funct. Anal. \textbf{22} (2012), no.~3, 621--626. \MR{2972603}

\bibitem[Bre21]{Br21}
\bysame, \emph{The isoperimetric inequality for a minimal submanifold in
  {E}uclidean space}, J. Amer. Math. Soc. \textbf{34} (2021), no.~2, 595--603.
  \MR{4280868}

\bibitem[CG92]{CG92}
Jaigyoung Choi and Robert Gulliver, \emph{The sharp isoperimetric inequality
  for minimal surfaces with radially connected boundary in hyperbolic space},
  Invent. Math. \textbf{109} (1992), 495--503.

\bibitem[FM19]{FM19}
Brian Freidin and Peter McGrath, \emph{Area bounds for free boundary minimal
  surfaces in a geodesic ball in the sphere}, J. Funct. Anal. \textbf{277}
  (2019), no.~11, 108276, 19. \MR{4013827}

\bibitem[FM20]{FM20}
\bysame, \emph{Sharp area bounds for free boundary minimal surfaces in
  conformally {E}uclidean balls}, Int. Math. Res. Not. IMRN (2020), no.~18,
  5630--5641. \MR{4153123}

\bibitem[GS87]{GS87}
Robert Gulliver and Peter Scott, \emph{Least area surfaces can have excess
  triple points}, Topology \textbf{26} (1987), no.~3, 345--359. \MR{899054}

\bibitem[Kla18]{K18}
Bo'az Klartag, \emph{Eldan's stochastic localization and tubular neighborhoods
  of complex-analytic sets}, J. Geom. Anal. \textbf{28} (2018), no.~3,
  2008--2027. \MR{3833784}

\bibitem[NZ22]{NZ22b}
Keaton Naff and Jonathan~J. Zhu, \emph{Moving monotonicity formulae for minimal
  submanifolds in constant curvature}, in preperation (2022).

\bibitem[Sim83]{Si83}
Leon Simon, \emph{Lectures on geometric measure theory}, Proceedings of the
  Centre for Mathematical Analysis, Australian National University, vol.~3,
  Australian National University, Centre for Mathematical Analysis, Canberra,
  1983. \MR{756417}

\bibitem[Zhu18]{Zhu18}
Jonathan~J. Zhu, \emph{Moving-centre monotonicity formulae for minimal
  submanifolds and related equations}, J. Funct. Anal. \textbf{274} (2018),
  no.~5, 1530--1552. \MR{3778682}

\bibitem[Zhu22]{Zhu22}
\bysame, \emph{Widths of balls and free boundary minimal submanifolds}, arXiv
  preprint arXiv:2203.10031 (2022).

\end{thebibliography}

\end{document}